\documentclass[12pt]{amsart}
\usepackage{amsmath, amsfonts,amssymb} 
\usepackage[usenames]{color}
\usepackage[all]{xy}
\usepackage[ top=1in, bottom=1in, left=1.25in, right=1.25in ]{geometry}
\usepackage[colorlinks=true, linkcolor=blue, citecolor=green, urlcolor=cyan, pagebackref]{hyperref}

\allowdisplaybreaks
\theoremstyle{plain}

\newtheorem{Teo}{Theorem}[section]

\newtheorem{Lemma}[Teo]{Lemma}
\newtheorem{Prop}[Teo]{Proposition}
\newtheorem{Cor}[Teo]{Corollary}

\theoremstyle{definition}
\newtheorem*{DefProp}{Proposition/Definition}
\newtheorem{Def}[Teo]{Definition}
\newtheorem{Rem}[Teo]{Remark}
\newtheorem{Question}[Teo]{Question}

\newtheorem{Notation}[Teo]{Notation}

\numberwithin{equation}{Teo}

\usepackage{color}   

\newcommand{\cK}{\mathcal{K}}

\newcommand{\cV}{\mathcal{V}}
\newcommand{\NN}{\mathbb N}
\newcommand{\ZZ}{\mathbb Z}
\newcommand{\FF}{\mathbb F}

\newcommand{\newLyu}{\widetilde{\lambda}}
\newcommand{\NewLyu}{\widetilde{\Lambda}}
\newcommand{\newLyuName}{\text{Lyubeznik numbers in mixed characteristic}}
\newcommand{\newLyuNameSing}{\text{Lyubeznik number in mixed characteristic}}
\newcommand{\newLyuNameA}{\text{Lyubeznik number}}
\newcommand{\newLyuNameB}{\text{in mixed characteristic}}

\newcommand{\End}{\operatorname{End}} 
\newcommand{\Ext}{\operatorname{Ext}} 
\newcommand{\Ass}{\operatorname{Ass}}
\newcommand{\Hom}{\operatorname{Hom}} 
\newcommand{\Dim}{\operatorname{dim}} 
\newcommand{\Coker}{\operatorname{Coker}}
\newcommand{\Ker}{\operatorname{Ker}}
\newcommand{\IM}{\operatorname{Im}}
\newcommand{\Ann}{\operatorname{Ann}}
\newcommand{\InjDim}{\operatorname{inj.dim}}

\newcommand{\Depth}{\operatorname{depth}}

\newcommand{\Spec}{\operatorname{Spec}}

\newcommand{\Length}{\operatorname{length}}

\newcommand{\LEX}{\operatorname{lex}}   
\newcommand{\Inf}{\operatorname{inf}}
\newcommand{\Char}{\operatorname{char}}
   
\newcommand{\connarrow}{\ar `[d] `[l] `[llld] `[ld] [lld]}

\newcommand{\FDer}[1]{\stackrel{#1}{\to}}

\newcommand{\surj}{\twoheadrightarrow}
\newcommand{\deriv}[2]{\frac{\partial^{#2}}{\partial {#1}^{#2}}}
\newcommand{\LC}[3]{H^{#1}_{#2}\left(#3\right)}

\begin{document}

\title{Lyubeznik numbers in mixed characteristic}
\author{Luis N\'u\~nez-Betancourt and Emily E.\ Witt}

\maketitle


\begin{abstract}

This manuscript defines a new family of invariants, analogous to the \emph{Lyubeznik numbers}, associated to \emph{any} local ring whose residue field has prime characteristic. 
In particular, as their nomenclature suggests, these \emph{$\newLyuNameA$s $\newLyuNameB$} are defined for all local rings of mixed characteristic.  
Some properties similar to those in equal characteristic
hold for these new invariants.  Notably, the ``highest" $\newLyuNameSing$ is a well-defined notion.
Although the $\newLyuName$ and their equal-characteristic counterparts
are the same for certain local rings of equal characteristic $p>0$, we also provide an example where they differ.

\end{abstract}

\section{Introduction}

\subsection{Background}

Let $I$ be an ideal of a regular local ring $S$ of equal characteristic.  When $S$ has prime characteristic, Huneke and Sharp proved that the Bass numbers of every local cohomology module of the form $H^i_I(S)$ are finite \cite{Huneke}.   Subsequently, Lyubeznik showed this for $S$ of equal characteristic zero using $D$-module theory.  He  also extended these finiteness results to a larger family of functors that include iterated local cohomology modules of the form $H_{I_1}^{i_1} \cdots H_{I_\ell}^{i_\ell}(S)$, for $I_1, \ldots, I_\ell$ ideals of $S$ and $i_1, \ldots, I_\ell \in \NN$.  Using these results, he defined families of invariants that are now called \emph{Lyubeznik numbers} \cite{LyuDmodules}.  

Suppose that $(R, m, K)$ is a local ring admitting a surjection from an $n$-dimensional local regular local ring $(S, \eta, K)$ containing a field.  If $I$ is the kernel of this surjection, the Lyubeznik numbers of $R$, depending on two nonnegative integers $i$ and $j$, are defined as $\lambda_{i,j} (R) := \Dim_K \Ext^i_S\left(K,H^{n-j}_I (S)\right)$.  Remarkably, these numbers only depend on the ring $R$ and on $i$ and $j$, not on $S$, nor even on the choice of surjection from $S$.  Moreover, if $R$ is \emph{any} local ring containing a field, letting $\lambda_{i,j} (R) := \lambda_{i,j} (\widehat{R})$ extends the definition, making the Lyubeznik numbers well defined for every such ring \cite{LyuDmodules}.  

For $R$ containing a field, the Lyubeznik numbers of $R$ provide essential information about the ring.  For example, if $d=\dim R$, then $\lambda_{i,j}(R) = 0$ for $j>d$, and $\lambda_{d,d}(R)\neq 0$ \cite[Properties 4.4i, 4.4iii]{LyuDmodules}.  If $R$ is a Cohen-Macaulay ring, then $\lambda_{d,d}(R) = 1$ \cite[Theorem 1]{Kawasaki}.  Moreover, the Lyubeznik numbers have extensive connections with geometry and topology, including \'etale cohomology and the connected components of certain punctured spectra.  (See, for example, \cite{B-B, GarciaSabbah, Kawasaki2, Walther2, W}.)

\subsection{Aim and Main Results}

The aim of this manuscript is to define a new invariant, analogous to the Lyubeznik numbers, which are associated to \emph{any} local ring whose residue field has prime characteristic.  In particular, these new invariants are defined for local rings of mixed characteristic.  Moreover, we study properties of the $\newLyuName$, and investigate when they agree with the equal-characteristic Lyubeznik numbers, as well as whether they ever differ from them.
 
If $S$ is a regular local ring of unramified mixed characteristic, then Lyubeznik, along with N\'u\~nez-Betancourt, have shown that the Bass numbers are also finite \cite{LyuUMC, Nunez}.  This enables us to give the following definition:

\begin{DefProp}[\ref{WDMC}, \ref{LyuNumsGen}]
Let $(R,m,K)$ be a local ring  such that $\Char(K)=p>0$, and let $\widehat{R}$ denote its completion. By the Cohen Structure Theorems, 
$\widehat{R}$ admits a surjective ring homomorphism $\pi : S\surj \widehat{R}$,  
where $S$ is an unramified regular local ring of mixed 
characteristic. Let $I=\Ker(\pi)$ and $n=\dim(S)$. Then, given integers $i,j \geq 0$, $\dim_K \Ext^{i}_S(K,H^{n-j}_{I} (S))$ is finite and depends only on $R$, $i$, and $j$.  
We define the \emph{$\newLyuName$ of $R$ with respect to $i$ and $j$} as
$$\newLyu_{i,j} (R):=\dim_K \Ext^{i}_S (K,H^{n-j}_{I} (S)) .$$ 
\end{DefProp} 

When $R$ is a ring of equal characteristic $p>0$, 
both the Lyubeznik numbers and the $\newLyuName$ are defined.
We give some conditions for which these two definitions agree.
However, in Theorem \ref{TeoCE}, we prove that these numbers are not the same in general. 

Motivated by analogous properties of the Lyubeznik numbers in equal characteristic, we study properties of these invariants (cf.\ \cite[Properties $4.4$]{LyuDmodules}).  Some similar vanishing properties hold (see Proposition \ref{PropertiesLyuMC}), as well as analogous computations for complete intersection rings (see Proposition \ref{PropCI}).
Moreover, the ``highest" $\newLyuNameSing$ of a local ring for which these invariants are defined
is a well-defined notion:  if $(R, m,K)$ is a local ring of dimension $d$ such that $\Char(R)=p>0$, $\newLyu_{i,j}(R) = 0$ if either $i>d$ or $j>d$ (see Theorem \ref{TeoInjDim} and Definition \ref{highest}).

We investigate conditions on a local ring $R$ of equal characteristic $p>0$ (i.e., when both the Lyubeznik numbers and the $\newLyuName$ of $R$ are defined) under which $\newLyu_{i,j}(R) = \lambda_{i,j}(R)$ for all $i,j\in \NN$.  We find that they agree for $R$ Cohen-Macaulay, and $R$ of dimension less than or equal to two (see Corollary \ref{Agree}).

Moreover, we give a specific example for which $\newLyu_{i,j}(R) \neq \lambda_{i,j}(R)$ for some $i,j\in\NN$, employing the work on Bockstein morphisms of Singh and Walther, as well as a computation of \`Alvarez Montaner and Vahidi \cite{Bockstein, LyuNumMontaner} (see Remark \ref{RmkDiff} and Theorem \ref{TeoCE}).

\subsection{Outline}  
We recall some definitions and properties of $D$-modules in Section \ref{Preliminaries}.  In addition, we introduce a functor that will play a fundamental 
role in proving that the $\newLyuName$ are well defined. In Section \ref{SecDef}, we define the $\newLyuName$, and prove some similar properties to those that hold for the Lyubeznik numbers in equal characteristic.  
Next, in Section \ref{HighestLyuNum}, we prove in Theorem \ref{TeoInjDim} that, with our definition, if $d = \dim R$, $\newLyu_{d,d}(R)$ is again the ``highest" Lyubeznik number.
The goal of Section \ref{SameNum} is to prove that the original Lyubeznik numbers and $\newLyuName$ are the same for certain rings of characteristic $p>0$. 
On the other hand, in Section \ref{DifferentNum}, we give an example where these numbers differ.

\section{Preliminaries}\label{Preliminaries} 

\subsection{$D$-modules}

Here, we briefly review some facts about $D$-modules that will be used; we refer readers to \cite{Bj2,Coutinho} for a more detailed exposition.

If $R \subseteq S$ is an extension of commutative rings, the \emph{$R$-linear differential operators of $S$}, denoted $D(S,R)$, is the (non-commutative) subring of $\End_R(S)$ defined inductively as follows:
\begin{itemize}
\item The \emph{operators of order zero} are the endomorphisms defined by multiplication by some element of $S$, and
\item $\phi \in \End_R(S)$ is an \emph{operator of order less than or equal to $\ell$} if $[\phi, \sigma] = \phi \cdot \sigma - \sigma \cdot \phi$ is an operator of order less than or equal to $\ell-1$ for every $\sigma \in S$.
\end{itemize}  
If $S = R[[x_1,\ldots,x_n]]$ for some $n \in \NN$, then \[D(S,R)=S\left \langle\frac{1}{t!} \deriv{x_i}{t} \ | \ t \in \NN, i = 1, \ldots, n \right \rangle  \subseteq \Hom_R(S,S)\]  \cite[Theorem 16.12.1]{EGA}.

The canonical map $M \to M_f$ is a map of $D(S,R)$-modules for any $D(S,R)$-module $M$, and any $f \in S$.  Since $S$ is a $D(S,R)$-module, this map induces a natural $D(S,R)$-module structure on any local cohomology module of the form $H^j_I(S)$, where $I$ is an ideal of $S$ and $j \in \NN$.  Moreover, it is not hard to see that this also makes $H^{j_\ell}_{I_s} \ldots H^{i_2}_{I_2}H^{i_1}_{I_1}(S)$ a $D(S,R)$-module as well \cite[Example 2.1(iv)]{LyuDmodules}.  Given a field $K$ and any $f \in S$, $S_f$ finite length in the category of $D(S,K)$-modules.  Thus, $H^{j_\ell}_{I_s} \ldots H^{i_2}_{I_2}H^{i_1}_{I_1}(S)$ does as well \cite[Corollary 6]{Lyu2}.

\subsection{A key functor} \label{KeyFunctorSection}
Lyubeznik introduced a certain functor in proving that the equal-characteristic Lyubeznik numbers are well defined (cf.\ \cite[Lemma $4.3$]{LyuDmodules}). 
In \cite[Section 2.2]{NWEqual}, the authors study further properties of this functor.  As we will rely greatly on some of these properties to prove that the $\newLyuNameA$s $\newLyuNameB$ are well defined, we include them here for the reader's convenience.  Please refer to \emph{loc.\ cit.} for all details and proofs.

Fix a Noetherian ring $R$, and let $S = R[[x]]$.  Let $G$ denote the functor from the category of $R$-modules to that of $S$-modules given by $G(-) = (-) \otimes_R S_x/S$.
Fix an $R$-module $M$.
Since $G(M)=M \otimes_R S_x/S = \bigoplus \limits_{\alpha \in \NN} \left(M \otimes R x^{-\alpha}\right)$, 
any $u \in G(M)$ can be written uniquely in the form $u= m_1 \otimes x^{-\alpha_1} + \ldots + m_\ell \otimes x^{-\alpha_\ell}$ for some $\ell, \alpha_i, m_i\in\NN$.  

Actually, $G$ is an equivalence between the category of $R$-modules to that of $D(S,R)$-modules supported on the $\cV \mathcal(xS)$, the Zariski closed subset of $\Spec (S)$ given by $xS$.
The $D(S,R)$-module structure of $G(M)$ is defined by the action of each $\frac{1}{t!} \deriv{x}{t}$, $t \in \mathbb{N}$:
$ \label{daction} \left(\frac{1}{t!}  \deriv{x}{t} \right)\cdot (m\otimes x^{-\alpha})=  \binom{\alpha+t-1}{t} \cdot \left((-1)^{t} m\otimes
x^{-\alpha-t}\right).$  The functor $G$ is also flat, and commutes with taking local cohomology.

For every ideal $I$ of $R$ and every $j \in\NN$, $G\left(H^j_I(M)\right)\cong H^{j+1}_{(I,x)S}(M\otimes_R S)$.  Moreover, if $R$ is Gorenstein, $P$ is a prime ideal of $R$, and $E_R(R/P)$ denotes the injective hull of $R/P$ over $R$, $G(E_R(R/P))=E_S (S/(P,x)S)$. In fact, in this case, $M$ is an injective $R$-module if and only if $G(M)$ is an injective $S$-module.  Using this fact, it can also be shown that, endowing $R=S/xS$ with the natural $S$-module structure via extension of scalars, if $N$ is also an $R$-module, $\Ext^{i}_S (M,G(N))=\Ext^{i}_R (M,N)$ for all $j \in \NN$.

\section{Definition and properties}\label{SecDef}
\begin{Lemma}\label{LemmaPowerSeries}
Let $S=R[[x]]$, where $(R,m,K)$ is a Gorenstein local ring.
Then for every ideal $I$ of $R$, and all $i, j\in \NN$, $$\dim_K\Ext^{i}_S(K,H^{j+1}_{(I,x)S} (S))=\dim_K\Ext^{i}_R(K,H^j_{I} (R)).$$
\end{Lemma}
\begin{proof}

Let $G$ be the functor defined in Section \ref{KeyFunctorSection}.  Now, $G(H^j_{I} (R)) = H^{j+1}_{(I,x)} (S)$ by \cite[Lemma 3.9]{NWEqual}. Since $R$ is Gorenstein, \cite[Proposition 3.12]{NWEqual} indicates that 
$\Ext^{i}_S (M,G(N))=\Ext^{i}_R (M,N)$ for all $R$-modules $M$ and $N$.  Thus, $\Ext^{i}_S(K,H^{j+1}_{(I,x)S} (S))=\Ext^{i}_R(K,H^j_{I} (R)),$ and we are done.

\end{proof}

\begin{Cor} \label{CorPowerSeries}
Let $(R,m,K)$ be a Gorenstein local ring, and let $S=R[[x_1, \ldots, x_n]]$ denote a power series ring over $R$. For every ideal $I$ of $R$ and all $i,j\in \NN$, $$\dim_K\Ext^{i}_S(K,H^{j+n}_{(I,x_1, \ldots, x_n)S} (S))=\dim_K\Ext^{i}_R(K,H^j_{I} (R)).$$
\end{Cor}

\begin{proof}
Using Lemma \ref{LemmaPowerSeries}, apply induction on $n$.
\end{proof}

\begin{Rem}[\cite{Cohen}] \label{Cohen} 
If $(V, pV, K)$ and $(W, pW, K)$ are both complete Noetherian DVRs with residue class field $K$, then $V$ and $W$ are isomorphic (via a non-unique isomorphism). 

By the Cohen Structure Theorems, if $(R, m, K)$ is a complete local ring of mixed characteristic $p>0$, $R$ is the homomorphic image of a ring of the form $V[[x_1, \ldots, x_n]]$, where $(V, pV, K)$ is an (unramified) mixed characteristic complete Noetherian discrete valuation domain.  Therefore, if $V[[x_1, \ldots, x_n]] \surj R$ and $W[[y_1, \ldots, y_{n'}]] \surj R$, then $V \cong W.$  (In fact, we can take $n = n'$ to be the embedding dimension of $R/pR$.)

If $(R,m,K)$ is a complete local ring of equal characteristic $p>0$, $R$ is a homomorphic image of some $K[[x_1, \ldots, x_n]]$ by the Cohen Structure Theorems.  Therefore, if $(V, pV, K)$ is the complete Noetherian DVR provided above, the composition $V[[x_1, \ldots, x_n]] \surj K[[x_1, \ldots, x_n]]\surj R$ is a surjection. 

Thus, any complete local ring $(R, m, K)$ such that $char(K) = p > 0$ is the homomorphic image of $V[[x_1, \ldots, x_n]]$, where $(V, pV, K)$ is a uniquely determined mixed characteristic complete Noetherian discrete valuation domain.
\end{Rem}

\begin{Prop}\label{WDMC}
Let $(R, m, K)$ be a local ring such that $char(K) = p>0$, admitting a surjection $\pi : S\surj R$, where $S$ is an unramified regular local ring of dimension $n$.
Let $I = \Ker(\pi)$, and take $i,j\in\NN.$  Then
$$\dim_K\Ext^{n-i}_S(K,H^j_{I} (S))$$ is finite and
depends only on $R$, $i$, and $j$, but not on $S$, nor on $\pi$.
\end{Prop}
\begin{proof}
Each $\dim_K\Ext^{i}_S(K,H^{n-j}_{I} (S))$ is finite (cf.\ \cite{LyuUMC,Nunez}), so it remains to prove that these numbers are well defined.
Let $\pi' : S'\surj R$ be another surjective map, where $S'$ is an unramified regular local ring of dimension $n'$.
Set $I' = \Ker(\pi')$, 
and let $m'$ be the maximal ideal of $S'$.
Since the Bass numbers with respect to the maximal ideal are not affected by completion, we may assume that
$R$, $S$, and $S'$ are complete. 
Let $V$ denote the complete DVR associated to $K$ given by the Cohen Structure Theorems (see Remark \ref{Cohen}), so that we can make identifications $S =  V[[x_1,\ldots,x_{n-1}]]$
 and $S'= V[[y_1,\ldots,y_{n'-1}]]$. 
 
Let $S''=V[[z_1,\ldots,z_{n+n'-2}]]$, and 
let $\pi'': S''\surj R$ be the surjective map defined by $\pi''(z_j)=\pi(x_j)$ for $1\leq j\leq n-1$
and $\pi''(z_j)=\pi'(y_{j-n+1})$ for $n\leq j\leq n+n'-2$.
Let $I''$ be the preimage of $I$ under $\pi''$, respectively. Let
 $\alpha:S\to S''$ be the map defined by $\alpha(x_j)=z_j$. Note that $\pi'' \circ \alpha=\pi.$
Since $\pi'$ is surjective, there exist $f_1,\ldots, f_{n'-1}\in S$ such that $\pi''(z_{n-1+j})=\pi(f_j)$ for $j\leq n'-1.$
Then $z_{n-1+j}-\alpha(f_j)\in\Ker(\pi'').$ We note that $\beta:S''\to S$, defined by sending
$z_j$ to $x_j$ for $j\leq n-1$, and $z_{n-1+j}$ to $f_j$ for $j\leq n'-1$, is a splitting of $\alpha.$
Then
 $I''=(I, z_{n}-\alpha(f_1),\ldots, z_{n'+n-2}-\alpha(f_{n'-1}))S''$. 
Since 
$$z_1,\ldots, z_{n-1}, z_{n}-\alpha(f_1),\ldots, z_{n'+n-2}-\alpha(f_{n'-1})$$ form a regular system of parameters, we obtain, by Corollary \ref{CorPowerSeries}, that 
$$
\dim_K\Ext^{i}_{S''} (K,H^{n+n'-j}_{I''}  (S''))=\dim_K\Ext^{i}_S (K,H^{n-j}_{I} (S)).
$$
By an analogous argument, we also have that
$$
\dim_K\Ext^{i}_{S''} (K,H^{n+n'-j}_{I''}  (S''))=\dim_K\Ext^{i}_{S'} (K,H^{n'-j}_{I'} (S')),
$$
and the result follows.
\end{proof}

\begin{Def}[$\newLyuName$] \label{LyuNumsGen}
Let $(R,m,K)$ be a local ring  such that $\Char(K)=p>0$.  By the Cohen Structure Theorems, the completion
$\widehat{R}$ admits a surjection $\pi : S\surj \widehat{R}$,  
where $S$ is an unramified regular local ring of mixed 
characteristic. Let $I=\Ker(\pi)$, $n=\dim(S)$, and $i,j\in\NN$. We define the \emph{$\newLyuNameSing$ of $R$ with respect to $i$ and $j$} as
$$\newLyu_{i,j} (R):=\dim_K\Ext^{i}_S(K,H^{n-j}_{I} (S)).$$ 
\end{Def}  

\noindent Note that by Proposition \ref{WDMC}, the $\newLyu_{i,j} (R)$ are well defined and depend only on $R$, $i$, and $j$.

\begin{Rem} In Definition \ref{LyuNumsGen}, we need to take the completion of $R$ for the Cohen Structure Theorems to ensure the existence of a surjection from an unramified regular local ring $S$ of mixed characteristic, $\pi : S\surj \widehat{R}$.
If such a map exists without taking the completion, then $\newLyu_{i,j}(R)=\dim_K\Ext^{i}_{\widehat{S}}(K,H^{n-j}_{I\widehat{S}} (\widehat{S}))=\dim_K\Ext^{i}_S(K,H^{n-j}_{I} (S))$, where $I=\Ker(\pi)$.
\end{Rem}

\begin{Rem} \label{DefnAgree}
Fix $(R,m,K)$, a local ring of equal characteristic $p>0$.  
On one hand, there exists a surjection from an $n$-dimensional unramified regular local ring of mixed characteristic, $\pi : S\surj \widehat{R}$. 
On the other hand, the induced map $\pi' : S/pS\surj \widehat{R}$ is also surjective.
If $I=\Ker(\pi)$ and $I'=\Ker(\pi')$,
 $I$ is the preimage of $I'$ under the canonical surjection
$S\surj S/pS$. 
There are two notions of Lyubeznik numbers corresponding to these homomorphisms, those given by Lyubeznik's original definition for rings of equal characteristic, $\lambda_{i,j}(R) = \dim_K\Ext^{i}_{S/pS} (K,H^{n-j-1}_{I'} (S/pS))$ \cite{LyuDmodules}, and the $\newLyuName$, $\newLyu_{i,j}(R) = \dim_K\Ext^{i}_S(K,H^{n-j}_{I} (S)).$ 
\end{Rem}

\noindent Remark \ref{DefnAgree} naturally incites the following question:

\begin{Question}\label{QuestionCharp}

Is $\newLyu_{i,j}(R) = \lambda_{i,j}(R)$ whenever both are defined, i.e., for every local ring $(R,m,K)$ of equal characteristic $p>0$ and all $i,j\in\NN$?  In other words, with notation as in Remark \ref{DefnAgree}, is it always true that $\dim_K\Ext^{i}_{S/pS} (K,H^{n-j-1}_{I'} (S/pS))=\dim_K\Ext^{i}_S(K,H^{n-j}_{I} (S))?$

\end{Question}
In Corollary \ref{Agree}, we prove an affirmative answer to Question \ref{QuestionCharp} for certain rings. However, 
there are cases in which the answer is negative:  Theorem \ref{TeoCE} presents such an example, a Stanley-Reisner ring over a field of characteristic two.

The $\newLyuName$ satisfy some similar vanishing properties to those that hold for the equal characteristic ones:

\begin{Prop} \label{PropertiesLyuMC}
Let $(R,m,K)$ be a local ring  such that $\Char(K)=p>0$ and $d=\dim(R)$. 
Then
\begin{itemize}
\item[\rm{(i)}] $\newLyu_{i,j} (R)=0 $ for $j>d$,
\item[\rm{(ii)}] $\newLyu_{i,j} (R)=0 $ for $i>j+1$, and
\item[\rm{(iii)}] $\newLyu_{d,d}(R)\neq 0$.
\end{itemize}
\end{Prop}
\begin{proof}
Let $\widehat{R}$ be the completion of $R$, so that $\widehat{R}$ admits a surjective ring homomorphism $\pi : 
S\surj \widehat{R}$, where $(S, \eta)$ is an unramified regular local ring of mixed characteristic and of dimension $n$.
Let $I=\Ker(\pi)$.

Property (i) holds since $H^{n-j}_I(S)=0$ for $j> \dim(S/I) = \dim R = d$, and (ii) holds since $\InjDim_S H^{n-j}_I(S) \leq \dim_S H^{n-j}_I(S) +1\leq i+1$ \cite{Zhou}. 

To prove (iii), first note that by an analogous argument to the proof of \cite[Property 4.4(iii)]{LyuDmodules}, $ H^d_\eta H^{n-d}_I(S) \neq 0.$   We will prove that $\newLyu_{d,d}(R)\neq 0$ by contradicting this fact. Suppose that $\newLyu_{d,d}(R)=\Ext^d_S(K,H^{n-d}_I(S))=0$. 

We claim that $\Ext^d_S(M,H^{n-d}_I(S))=0$ for every finite-length $S$-module $M$. We will prove this by induction on 
$h=\Length_S(M)$. If $h=1$, then $M=K$, and the statement holds by assumption. Suppose that the statement is true for all $N$ with $\Length_S N < h+1$, and take $M$ with $\Length_S M =h+1$. Then there exists a short exact sequence $0\to K\to M\to M'\to 0$, 
where $M'$ is an $S$-module of length $h$.  The long exact sequence in Ext gives:
$$
\cdots\to\Ext^d_S(M',H^{n-d}_I(S))\to \Ext^d_S(M,H^{n-d}_I(S))\to \Ext^d_S(K,H^{n-d}_I(S))\to\cdots.
$$
Since $\Ext^d_S(K,H^{n-d}_I(S))=\Ext^d_S(M',H^{n-d}_I(S))=0$ by the inductive hypothesis, $\Ext^d_S(M,H^{n-d+1}_I(S))=0$, and the claim follows.

This claim implies that $\Ext^d_S(S/\eta^\ell,H^{n-d}_I(S))=0$ for all $\ell \geq 1$. 
Then $H^d_\eta H^{n-d}_I(S)=\lim\limits_{\overset{\longrightarrow}{\ell}} \Ext^d_S(S/\eta^\ell,H^{n-d}_I(S))=0$, the sought contradiction.
\end{proof}

\begin{Prop} \label{PropCI}
Let $(V,pV,K)$ be an complete DVR of unramified mixed characteristic $p>0$, and let $S=V[[x_1,\ldots, x_n]]$.
Let $f_1,\ldots, f_\ell\in S$ be a regular sequence. Then
$$\newLyu_{i,j} \left(S/(f_1,\ldots, f_\ell) \right)=1$$ for $i=j=n+1-\ell$, and vanishes otherwise.
\end{Prop}
\begin{proof}

Our proof will be by induction on $\ell$. If $\ell=1$, 
we have the short exact sequence
$$
0\to S\to S_{f_1} \to H^1_{f_1S}( S)\to 0.
$$
Then $\Ext^i_S(K, S)\cong \Ext^{i+1}_S(K, H^1_{f_1S}( S))$ for every $i\geq 0$ because $\Ext^i_S(K, S_f)=0.$

Suppose that the formula holds for $\ell-1$ and we will prove it for $\ell$.   From the exact sequence
$$
0\to H^{n-\ell-1}_{(f_1,\ldots, f_{\ell-1})S} (S)\to H^{n-\ell-1}_{(f_1,\ldots, f_{\ell-1})S}(S_{f_{\ell}})
\to H^{n-\ell}_{(f_1,\ldots, f_\ell)S}(S)\to  0,
$$
we obtain that $\Ext^i_S(K, H^{n-\ell}_{(f_1,\ldots, f_\ell)S}(S))=\Ext^{i+1}_S(K, H^{n-\ell-1}_{(f_1,\ldots, f_{\ell-1})S} (S))$ 
for every $i\geq 0$ because $\Ext^i_S(K, H^{n-\ell}_{(f_1,\ldots, f_\ell)S}(S_{f_{\ell+1}}))=0$.
\end{proof}

\section{Existence of the highest Lyubeznik Number in Mixed Characteristic}\label{HighestLyuNum} 

\begin{Lemma}\label{LemmaHom}
Let $(V,pV,K)$ be a complete DVR of unramified mixed characteristic $p>0$, and let $S=V[[x_1\ldots, x_n]]$.
Then $$\End_{D(S,V)}\left(E_S(K)\right)=V.$$
\end{Lemma}
\begin{proof}
Let $\phi\in\End_{D(S,V)}\left(E_S(K)\right) \subseteq \End_{S}\left(E_S(K)\right) =S$;
 $\phi$ must correspond to multiplication by some $r\in S$. Thus, $\partial (rw)=r\partial (w)$ for every $w\in E_S(K)$ and 
$\partial\in D(S,V)$.
We will prove that $r\in V$ by contradiction. If $r\not\in V$, we may assume there exists $\alpha=(\alpha_1,\ldots,\alpha_n)\in \NN^n\setminus \{(0,\ldots,0)\}$  such that 
$$r=a+bx^{\alpha} + \sum_{\beta\in\NN, \beta\geq_{{\tiny \LEX}} \alpha} c_\beta x^\beta,$$ 
where $a,b,c_\beta \in V$ and $b\neq 0$.
Then for every $j\in \NN$,
$$
r\frac{(-1)^{\alpha_1-1}}{\alpha_1!}\frac{\partial^{\alpha_1}}{\partial x_1 ^{\alpha_1}}\cdots \frac{(-1)^{\alpha_n-1}}{\alpha_n!}\frac{\partial^{\alpha_n}}
{\partial x_n ^{\alpha_n}} \frac{1}{p^j x_1\cdots x_n}=
 \frac{r}{p^j x^{\alpha_1+1}_1\cdots x^{\alpha_n+1}_n}
$$
$$
=\frac{a}{p^j x^{\alpha_1+1}_1\cdots x^{\alpha_n+1}_n}  +\frac{b}{p^j x_1\cdots x_n}
$$
On the other hand, for every $j \in \NN$,
\begin{align*}
&\frac{(-1)^{\alpha_1-1}}{\alpha_1!}\frac{\partial^{\alpha_1}}{\partial x_1^{\alpha_1}}\cdots \frac{(-1)^{\alpha_n-1}}{\alpha_n!}\frac{\partial^{\alpha_n}}{\partial x_n^{\alpha_n}} \frac{r}{p^j x_1\cdots x_n}\\
=&
\frac{(-1)^{\alpha_1-1}}{\alpha_1!}\frac{\partial^{\alpha_1}}{\partial x_1^{\alpha_1}}\cdots \frac{(-1)^{\alpha_n-1}}{\alpha_n!}\frac{\partial^{\alpha_n}}{\partial x_n^{\alpha_n}}
\frac{a}{p^j x_1\cdots x_n}\\
=&
\frac{a}{p^j x^{\alpha_1+1}_1\cdots x^{\alpha_n+1}_n}
\end{align*}
Then
$$
\frac{a}{p^j x^{\alpha_1+1}_1\cdots x^{\alpha_n}_n}  +\frac{b}{p^j x_1\cdots x_n}=
\frac{a}{p^j x^{\alpha_1+1}_1\cdots x^{\alpha_n}_n},
$$
so $b\in p^j V$ for every $j\in \NN$. This means that $b=0$, a contradiction. Thus, $r\in V$.
Since every map given by multiplication by an element in $V$ is already a map of $D(S,V)$-modules,
we are done. 
\end{proof}

\begin{Prop}\label{PropAnn}
Let $(V,pV,K)$ be a complete DVR of unramified mixed characteristic $p>0$, and let $S=V[[x_1\ldots, x_n]]$.
Let $N\subsetneq E_S(K)$ be a proper $D(S,V)$-submodule. Then $N=\Ann_{E_S(K)} p^\ell S$ for some $\ell\in\NN$.
\end{Prop}
\begin{proof}
Let $v\in N$ be such that $v\in \Ann_{E_S(K)} p^\ell S$ but $v\not\in \Ann_{E_S(K)} p^{\ell-1}  S$.
We claim that $D(S,V)\cdot v=\Ann_{E_S(K)} p^{\ell}  S$ by induction on $\ell$.
If $\ell=1$, $\Ann_{E_S(K)} p  S=E_{S/pS}(K)$, a simple $D(S,K)$-module, and the claim holds.
Now, we suppose the claim true for $\ell-1$.
Since $\Ann_{E_S(K)} p^{\ell}  S /\Ann_{E_S(K)} p^{\ell-1} S=E_{S/pS} (K)$,  there exists an operator $\partial\in D(S,V)$ 
such that 
$$ \partial v= 1/p^\ell x_1\cdots x_n+ w$$ 
for an element $w\in \Ann_{E_S(K)} p^{\ell-1}  S$. 
Then  $p \partial v\in \Ann_{E_S(K)} p^{\ell-1}  S \setminus \Ann_{E_S(K)} p^{\ell-2}  S$. Thus, there exists 
an operator $\delta$ such that $p\delta\partial v=w$ by the induction hypothesis, so $ 1/p^\ell x_1\cdots x_n=(\partial -p\delta \partial)v$.
Therefore,   $\Ann_{E_S(K)} p^{\ell}  =D(S,V)\cdot 1/p^\ell x_1\cdots x_n \subseteq D(S,V)\cdot v\subseteq \Ann_{E_S(K)} p^{\ell}$, proving our claim.

Since $N\neq E_S(K)$, $\ell=\Inf\{j\in\NN\mid 1/p^j x_1\cdots x_n \in N\}$ is a natural number.
Hence, $N=D(S,V)\cdot 1/p^\ell x_1\cdots x_n=\Ann_{E_S(K)} p^\ell S$.
\end{proof}

\begin{Lemma}\label{LemmaSub}
Let $(V,pV,K)$ be a complete DVR of unramified mixed characteristic $p>0$, and let $S=V[[x_1\ldots, x_n]]$.
Let $M\subseteq \bigoplus \limits^{h}_{i=1} E_S(K)$ be a $D(S,V)$-submodule.
Then
$M\FDer{\cdot p} M$ is surjective if and only if $M$ is an injective $S$-module.
\end{Lemma}
\begin{proof}

Suppose that $M$ is an injective $S$-module. Since $E_S(K)\FDer{\cdot p} E_S(K)$ is surjective and $M=\bigoplus \limits_\ell E_S(K)$,  
$M\FDer{\cdot p} M$ is also surjective.

Now assume that $M\FDer{\cdot p} M$ is surjective. We will show that $M$ is injective by contradiction. Suppose that $M\neq E_S(M)$. 
As $M$ is a $D(S,V)$-module supported only at the maximal ideal,
$\InjDim(M)\leq 1$ (cf. \cite{Nunez,Zhou}).
Let $0\to M\to E_1\FDer{\phi} E_2\to 0$ be a minimal injective resolution of $M$; in particular, $E_2\neq 0$.
Let $\mu_i=\Ext^i_S(K,M)$. 
Now, $\mu_1$ is finite and less than or equal to $h$. Let $(-)^*= \Hom_S (-,E_S(K))$ be the Matlis duality functor. 
From the short exact sequence
$0\to E^*_2\FDer{\phi^*} S^{\mu_1}\to M^*\to 0$, we obtain that $E^*_2$ is a free module of finite rank less than or equal to $\mu_1$, so, 
$E_2=  \bigoplus \limits^{\mu_2}_{i=2} E_S(K)$. By Lemma \ref{LemmaHom},  $\phi$ is given by a 
$\mu_1\times \mu_2$-matrix with entries in $V$. Thus, $\phi^*: S^{\mu_2}\to S^{\mu_1}$ can be represented as a matrix by the 
transpose of a matrix that represents $\phi$.
We may consider $\phi^*$ as a map of free $V$-modules, $\phi^*: V^{\mu_2}\to V^{\mu_1}$. 
By the structure theorem for finitely-generated modules over a principal ideal domain, there are isomorphisms 
$\varphi_1: V^{\mu_1}\to V^{\mu_1}$ and $\varphi_2: V^{\mu_2}\to V^{\mu_2}$, such that 
$\varphi_1 \phi^* \varphi_2$ is a matrix whose entries are zero off the diagonal. That is, we have
$$
\xymatrix{
 V^{\mu_2}  \ar[rr]^{\phi^*}  & & V^{\mu_1} \ar[d]^{\varphi_1}\\
 V^{\mu_2}  \ar[rr]^{\varphi_1 \phi^* \varphi_2} \ar[u]^{\varphi_2} && V^{\mu_1} \\
}
$$
 
Let $a_1,\ldots a_{\mu_1}\in V$ be the elements on the diagonal of $\varphi_1 \phi^* \varphi_2$, and let 
$v:V\to \NN$ be the valuation.
Since $E_S(M)=E_1\to E_2$ is surjective,
none of $a_1,\ldots, a_{\mu_1}$ is zero.
Since $E_2\neq 0$ and the injective resolution 
$0\to M\to E_1\FDer{\phi} E_2\to 0$ is minimal,
 none of $a_1,\ldots, a_{\mu_1}$ are units in $V$.
 Then $a_1,\ldots, a_{\mu_1}\in pV\setminus \{0\}$.
We extend $\varphi_i$ as a isomorphism of $S$-modules, $\varphi_i: S^{\mu_i}\to S^{\mu_i}$.
Then $\varphi^*_i: E_i\to E_i$ is an isomorphism of $D(S,V)$-modules. We obtain the following commutative diagram

$$
\xymatrix{
0\ar[r] & M\ar[r]   & E_1  \ar[rr]^{\phi} & & E_2 \ar[r] \ar[d]^{\varphi^*_2} & 0\\
0\ar[r] & \Ker(\varphi^*_2 \phi \varphi^*_1) \ar[r] \ar[u]^{\varphi^*_1}   & E_1  \ar[rr]^{\varphi^*_2 \phi \varphi^*_1}  \ar[u]^{\varphi^*_1}  && E_2  \ar[r] & 0\\
}
$$
Therefore, $M\cong \Ker(\varphi^*_2 \phi \varphi^*_1)=\bigoplus \limits^{\mu_1-\mu_2}_{i=1} E_{S}(K) + \bigoplus \limits^{\mu_2}_{i=1}\Ann_{E_S(K)} p^{v(a_i)}S$
which is a contradiction because 
{\small
$$\bigoplus^{\mu_1-\mu_2}_{i=1} E_{S}(K) + \bigoplus^{\mu_2}_{i=1}\Ann_{E_S(K)} p^{v(a_i)}S\FDer{\cdot p} \bigoplus^{\mu_1-\mu_2}_{i=1} E_{S}(K) + \bigoplus^{\mu_2}_{i=1}\Ann_{E_S(K)} p^{v(a_i)}S$$} is not surjective.
\end{proof}

\begin{Lemma}\label{LemmaOne}
Let $(V,pV,K)$ be a complete DVR of unramified mixed characteristic $p>0$, and let $S=V[[x_1\ldots, x_n]]$.
Let $m$ denote the maximal ideal of $S$, and let $I\subseteq S$ be an ideal such that $\dim(S/I)=1$. Then $H^0_m H^n_I (S) = 0$ and $H^1_m H^n_I (S) \cong E_R(K)$.
\end{Lemma}

\begin{proof}
Let $f\in m$ be an element not in any minimal prime of $I$, so $\sqrt{I+fS}=m$. We have the following short exact sequence
$$
0\to H^{n}_I(S)\to H^{n}_I(S_f)\to H^{n+1}_{I + fS}(S) \cong E_S(K)\to 0.
$$
Since $f \in m$ and $H^{n}_I(S_f) \cong H^{n}_I(S)_f$, $H^0_m H^{n}_I(S_f) = 0$, which implies that $H^0_m H^n_I (S)= 0$.
Moreover, $H^1_m H^n_I (S)=E_S(K)$.
\end{proof}

\begin{Lemma}\label{LemmaDimPureTwoInj}
Let $(V,pV,K)$ be a complete DVR of unramified mixed characteristic $p>0$, and let $S=V[[x_1\ldots, x_n]]$.
Let $I\subseteq S$ be an ideal of pure dimension $2$. Then $H^{n}_I(S)$ is an injective $S$-module supported only at the maximal ideal.
\end{Lemma}

\begin{proof}
Let $R$ denote $S/pS$.
The short exact sequence $0\to S\FDer{\cdot p} S\to R \to 0$ induces the long exact sequence 
$\cdots \to H^{n}_{I} (S) \FDer{\cdot p} H^{n}_{I} (S) \to H^{n}_{I} (R)\to 0,$
where $H^{n}_{I} (R)=0$ by the Hartshorne-Lichtenbaum Vanishing Theorem, as $\sqrt{I+pS}\neq m$.
Thus, $H^{n}_{I} (S)\FDer{\cdot p} H^{n}_{I} (S)$ is surjective.
Now, 
$H^{n}_{IS_P} (S_P)=0$ for every prime ideal $P$ not containing $I$.  
If $I\subseteq P$ and $P\neq m,$ then $\dim(S/P)=1$ and
$H^{n}_{IR_P} (R_P)=0$
by the Hartshorne-Lichtenbaum Vanishing Theorem because $I$ has pure dimension $2$ and $\sqrt{IS_P}\neq PS_P$.
Therefore, $H^{n}_{I} (R)$ is a $D(S,V)$-module supported only at the maximal ideal.
Since $\dim_K \Ext^0_S(K,H^n_I(S))$ is finite, $H^n_I(S)$ is injective by Lemma \ref{LemmaSub}.
\end{proof}

\begin{Lemma}\label{LCVPureTwo}
Let $(V,pV,K)$ be a complete DVR of unramified mixed characteristic $p>0$, and let $S=V[[x_1\ldots, x_n]]$.
Let $m$ denote the maximal ideal of $S$, and let $I\subseteq S$ be an ideal of pure dimension two. Then $H^0_m H^{n-1}_I (S)=H^1_m H^{n-1}_I (S)=0$.  Moreover, $H^{n}_I(S) \cong E_S(K)^{\bigoplus \alpha}$ for some $\alpha \in \NN$, and $H^{2}_m H^{n-1}_I(S) \cong E_S(K)^{\bigoplus \alpha+1}$.  In particular, $H^2_m H^{n-1}_I (S) $ is an injective $S$-module.
\end{Lemma}
\begin{proof}
Let $f\in m$ be an element not in any minimal prime of $I$. Then $\sqrt{I+fS}\neq m$. Applying the Hartshorne-Lichtenbaum Vanishing Theorem, since $H^{n}_I(S)$ is supported at $m$ by Lemma \ref{LemmaDimPureTwoInj}, we obtain the exact sequence
$$
0\to H^{n-1}_I(S)\to H^{n-1}_I(S_f)\to H^{n}_{I+fS}(S)\to H^{n}_I(S)\to 0.
$$
Splitting the sequence into two short exact sequences, we obtain
\begin{align*}
&0\to H^{n-1}_I(S)\to H^{n-1}_I(S_f)\to M\to 0, \text{ and}\\
&0\to M \to H^{n}_{I+fS}(S)\to H^{n}_I(S)\to 0.
\end{align*}

\noindent These induce the following long exact sequences: 
\[ \xymatrix@R=.65cm{ \ar[r] 
0 & H^{0}_m H^{n-1}_I(S) \ar[r] & H^{0}_m H^{n-1}_I(S_f) \ar[r] & H^{0}_m(M) \connarrow & \\ 
& H^{1}_m H^{n-1}_I(S) \ar[r] & H^{1}_m H^{n-1}_I(S_f) \ar[r] & H^{1}_m(M)   \connarrow & \\ 
& H^{2}_m H^{n-1}_I(S) \ar[r] & H^{2}_m H^{n-1}_I(S_f) \ar[r] & H^{2}_m(M)   \ar[r] & 0, 
}\] and
\[ \xymatrix@R=.65cm{ \ar[r] 
0 & H^{0}_m( M)\ar[r] & H^{0}_m H^{n}_{I+fS}(S) \ar[r] & H^{0}_m H^{n}_I(S)  \connarrow & \\ 
& H^{1}_m( M)  \ar[r] & H^{1}_m H^{n}_{I+fS}(S) \ar[r] & H^{1}_m H^{n}_I(S)  \connarrow & \\ 
& H^{2}_m( M) \ar[r] & H^{2}_m H^{n}_{I+fS}(S) \ar[r] & H^{2}_m H^{n}_I(S) \ar[r] & 0.
}\]


Since all $H^{j}_m H^{n-1}_I(S_f) =0$, we know that
$H^{0}_m  H^{n-1}_I(S) =H^{2}_m( M)=0$. Since $\dim(S/(I+fS))=1,$
$H^{0}_m  H^{n}_{I+fS}(S)=H^{2}_m H^{n}_{I+fS}(S)=0$ by Lemma \ref{LemmaOne}, which implies both that $H^{0}_m(M)=H^{1}_m H^{n-1}_I(S)=0$ and that $\LC{1}{m}{M} \cong H^{2}_{m} H^{n-1}_{I}(S)$.
In addition, $H^{1}_m H^{n}_{I}(S)=H^{2}_m H^{n}_{I}(S)=0$
by Lemma \ref{LemmaDimPureTwoInj}.
Thus, we have a short exact sequence
$$
0\to H^{0}_m( H^{n}_I(S))\to
 H^{2}_m H^{n-1}_I(S) \to H^{1}_m H^{n}_{I+fS}(S) \to 0.
$$

By Lemma \ref{LemmaDimPureTwoInj}, $H^{n}_I(S)$ is an injective $S$-module supported only at $m$, and its Bass numbers are finite by \cite{LyuUMC, Nunez}, so $H^0_m H^n_I(S) = H^n_I(S) \cong E_R(K)^{\bigoplus \alpha}$ for some $\alpha \in \NN.$   Moreover, by Lemma \ref{LemmaOne},  $H^{1}_m H^{n}_{I+fS}(S) \cong E_S(K)$.

Thus, we have the short exact sequence \[0 \to E_S(K)^{\bigoplus \alpha} \to H^{2}_m H^{n-1}_I(S) \to E_S(K) \to 0,\] which splits, so that $H^{2}_m H^{n-1}_I(S) \cong E_S(K)^{\bigoplus \alpha+1}$.
\end{proof}

\begin{Cor}
Let $(V,pV,K)$ be a complete DVR of unramified mixed characteristic $p>0$, and let $S=V[[x_1\ldots, x_n]]$.
Let $I$ be an ideal of $S$ of pure dimension two. Then $H^j_Q H^i_I(S_Q)$ is injective for every prime ideal $Q$ of $S$.
\end{Cor}
\begin{proof}
This follows from Lemmas \ref{LemmaOne} and \ref{LCVPureTwo}.
\end{proof}

\begin{Lemma}\label{LCVTwo}
Let $(V,pV,K)$ be a complete DVR of unramified mixed characteristic $p>0$, and let $S=V[[x_1\ldots, x_n]]$.
Let $I$ be an ideal of $S$ such that $\dim(S/I)=2$, and let $m$ denote its maximal ideal.
 Then $H^0_m H^{n-1}_I (S) =H^1_m H^{n-1}_I (S)=0$
and $H^2_m H^{n-1}_I (S) $ is an injective $S$-module.
\end{Lemma}
\begin{proof}
Let $J_1$ and $J_2$ be two ideals of pure dimensions $1$ and $2$, respectively, such that $I=J_1\cap J_2$. Using the Mayer-Vietoris sequence of local cohomology,
we obtain that $H^{n-1}_{I}(S)=H^{n-1}_{J_2}(S)$. Thus, for all $j$, $H^{j}_m (H^{n-1}_{I}(S))=H^j_m(H^{n-1}_{J_2}(S))$, and the result follows 
by Lemma \ref{LCVPureTwo}.
\end{proof}

\begin{Prop}\label{PropInj}
Let $(V,pV,K)$ be a complete DVR of unramified mixed characteristic $p>0$, and let $S=V[[x_1\ldots, x_n]]$.
Let $I$ be an ideal of $S$ such that $\dim(S/I)=d$, and let $m$ denote its maximal ideal.
Then $H^{d}_m H^{n-d+1}_I (S)$ is an injective  $S$-module, and $H^{j}_m H^{n-d+1}_I (S)=0$ for $j> d$.
\end{Prop}
\begin{proof}
We proceed by induction on $d$. If $d=0,1,$ or 2, we have the result by Lemmas \ref{LemmaOne} and \ref{LCVTwo}.
Suppose that $d\geq 3$ and the statement holds for $d-1$.
If $\Ass_S H^{n-d}_I(S) \neq \{ m \},$ we pick an element $r\in m$ that is neither in any minimal prime of $I$, nor of $H^{n-d}_I(S)$, which is possible
because $\Ass_S H^{n-d}_I(S)$ is finite (cf. \cite{LyuUMC,Nunez}).
On the other hand, $\Ass_S H^{n-d}_I(S) = \{ m\},$ we pick an element $r\in m$ not in any minimal prime of $I$. We have that 
$H^{d}_m (H^{n-d+1}_I(S)=H^{d-1}_m H^{n-d+2}_{I+rS}(S)$ and $H^{j}_m H^{n-d+1}_I(S)=H^{j-1}_m H^{n-d+2}_{I+rS}(S)=0$ 
for $j>d$ as in the proof of \cite[Proposition $2.1$]{W} because the conclusions
of in \cite[Lemmas 2.3 and 2.4]{W} hold in our case. Hence, the result follows by the induction hypothesis.
\end{proof}

\begin{Teo}\label{TeoInjDim}
Let $(S,m,K)$ be either a regular local ring of unramified mixed characteristic, or a regular local ring containing a field.
Let $n= \dim(S)$, and let $I$ be an ideal of $S$ such that $\dim(S/I)=d$. Then
$\InjDim H^{n-d}_I(S)= d$. 
\end{Teo}
\begin{proof}
We need to prove that $\Ext^j(R_Q/QR_Q, H^i_{IR_Q} (R_Q))=0$ for every prime ideal $Q$ of $R$, all $i \in \NN$, and all $j>d.$
We may assume that $Q$ is $m$, the maximal ideal of $S$, because 
if $Q \subsetneq m$, then $\dim R_Q/IR_Q<d$ and 
$
\InjDim_{R_Q} H^i_{IR_Q} (R_Q)\leq \dim_{R_Q} H^i_{IR_Q} (R_Q)\leq d
$ 
by \cite[Theorem $5.1$]{Zhou}.

We proceed by induction on $n$.
If $n=0$, $S$ is a field and the result follows.
Assume that the statement holds for all such $S$ of dimension less than $n$.

Since the theorem is already true for regular local rings that contain a field (cf.\ \cite{Huneke,LyuDmodules,LyuInjDim}), we will focus
on the case where  $S$ is an unramified regular local ring of unramified mixed characteristic.

Let $E^*=E^1\to E^2\to\ldots $ be a minimal injective resolution for $H^{n-d+1}_I(S)$. 
By \cite[Theorem 5.1]{Zhou}, $E^j=0$ for $j>d+1$.  
For every prime ideal $Q\subseteq S$, $S_Q$ is
either an unramified regular local ring of mixed characteristic or a regular local ring containing a field. 
Moreover, $\dim(S_Q/IS_Q)\leq d-1$ for every prime ideal $Q \subsetneq m$. 
Thus, $(E^d)_Q=(E^{d+1})_Q=0$ by the inductive hypothesis.
Hence, $E^d$ and $E^{d+1}$ are supported only at $m$.

Let $M=\IM(E^{d-1}\to E^{d})=\Ker(E^{d}\to E^{d+1})$. It suffices prove that $M$ is an injective $S$-module. The modules 
$H^j_m H^{n-d}_I(S)$ can be computed from the complex
$$H^0_m(E^*)=H^0_m(E^1)\to H^0_m( E^2)\to\ldots $$
Let $B^j=\IM\left(H^0_m(E^{j-1})\to H^0_m(E^{j})\right)$ and 
$Z^j=\Ker\left(H^0_m(E^{j})\to H^0_m(E^{j+1})\right)$. Note that $Z^d=M$ since $E_d$ and $E_{d+1}$ are supported only at $m$.
Since $\InjDim Z^j \leq 1$ and $\InjDim H^j_m H^{n-d}_I(S) \leq 1$ by the proof of \cite[Theorem 5.1]{Zhou} or by
\cite[ Theorem $0.3$ ]{Nunez},
as in the proof of \cite[Theorem 5.1]{Zhou}, we obtain that $B^j$ is injective
from the following short exact sequences:
\begin{align*}
&0\to  Z^j \to H^0_m(E^j)\to B^j\to 0, \text{ and} \\
&0\to B^{j-1} \to Z^j\to H^j_m(H^{n-d}_I(S))\to 0.
\end{align*}
Since $H^{d}_m H^{n-d}_I(S)$ injective by Proposition \ref{PropInj}, we know that $Z^d=M$ is injective due to the short exact sequence
$
0\to B^{d-1} \to Z^d\to H^j_m H^{n-d}_I(S)\to 0.
$
Therefore, $E_{d+1}=0$, so $\InjDim H^{n-d}_I(S)=d$. 
\end{proof}

\begin{Def}[Highest $\newLyuNameSing$] \label{highest}
For $(R,m,K)$ a local ring of dimension $d$ such that $\Char(K)=p>0$,
the \emph{highest}  \emph{$\newLyuNameA$} \emph{$\newLyuNameB$} \emph{of $R$} is defined as
$\newLyu_{d,d} (R).$ 
\end{Def}

Note that the nomenclature ``highest" is justified since $\newLyu_{i,d}(R)=0$ for $i>d$ by Theorem \ref{TeoInjDim}.
Moreover, we may also justify the following definition:

\begin{Def}[Lyubeznik table in mixed characteristic]
For $(R,m,K)$ a local ring such that $\Char(K)=p>0$ and $d=\dim(R)$, the \emph{Lyubeznik table in mixed characteristic} is the $(d+1) \times (d+1)$ matrix $\NewLyu(R)$, where $\NewLyu(R)_{i,j} = \newLyu_{i,j}(R)$ for $0\leq i, j\leq d$.
\end{Def}

\begin{Rem}
Recall that for a local ring $R$ of dimension $d$ containing a field, the \emph{Lyubeznik table} of $R$ is defined as the $(d+1) \times (d+1)$ matrix $\NewLyu(R)$ such that $\NewLyu(R)_{i,j} = \lambda_{i,j}(R)$ for $0 \leq i,j \leq R$. This matrix contains all nonzero Lyubeznik numbers, and is also upper triangular, since $\lambda_{i,j}(R) = 0$ if either $i>j$ or $j>d$ \cite[Properties 4.4i, 4.4ii]{LyuDmodules}.

On the other hand, Proposition \ref{PropertiesLyuMC} and Theorem \ref{TeoInjDim} imply that the Lyubeznik table in mixed characteristic contain all nonzero Lyubeznik numbers in mixed characteristic.  However, Proposition \ref{PropertiesLyuMC}
only implies that the Lyubeznik table in mixed characteristic is nonzero below the subdiagonal.  
\end{Rem}

\section{Examples in characteristic $p>0$ where the equal-characteristic and the $\newLyuName$ are equal}\label{SameNum}

\begin{Lemma}\label{LemmaExtSurj}
Let $(V,pV,K)$ be a complete DVR of unramified mixed characteristic $p>0$, and let $S=V[[x_1\ldots, x_n]]$.
Let $M$ be an $S$-module such that $\dim_K\Ext^i_S(K,M)$ is finite for all $i\in \NN$. 
Suppose that $M\FDer{\cdot p} M$ is surjective. Then for all $i\in \NN$,
$$
\dim_K\Ext^i_S(K,M)=\dim_K\Ext^i_{S/pS} (K,\Ann_M pS).
$$
\end{Lemma}

\begin{proof}
Let $R=S/pS$ and $N=\Ann_M (pS)$. The short exact sequence $0\to N\to M\FDer{\cdot p}M\to 0$
induces a long exact sequence
$$
0\to \Ext^0_S(K,N) \to \Ext^0_S(K,M)\FDer{\cdot p}\Ext^0_S(K,M)\to\Ext^1_S(K,N)\to\cdots.
$$
Since multiplication by $p$ is zero on $\Ext^i_S(K,M)$, we have 
short exact sequences
$$
0\to\Ext^{i-1}_S(K,M) \to \Ext^i_S(K,N) \to \Ext^i_S(K,M)\to 0.
$$
for all $i \in \NN$.  Thus,
$$
\dim_K\Ext^i_S(K,N)=\dim_K\Ext^{i-1}_S(K,M)+\dim_K\Ext^i_S(K,M).
$$
We can compute $\Ext^i_{S} (K,N)$ using the Koszul complex, $\cK$,
with respect to the 
sequence $p,x_1,\ldots,x_n$ in $S$. On the other hand,  we 
can compute $\Ext^i_{R} (K,N)$ using the Koszul complex, $\overline{\cK}$,
with respect to the 
sequence $\overline{x}_1,\ldots,\overline{x}_n$, in $R$.
Now, $\cK(N)$ is the direct sum of $\overline{\cK}(N)$ and an indexing shift  of the same 
complex by one. 
This means that 
\begin{align*}
\dim_K\Ext^i_S(K,N)=\dim_K\Ext^{i-1}_R(K,N)+\dim_K\Ext^i_R(K,N), \text{ so} 
\end{align*}
\begin{align*}
\dim_K\Ext^{i-1}_S(K,M)+\dim_K\Ext^i_S(K,M)=
\dim_K\Ext^{i-1}_R(K,N)+\dim_K\Ext^i_R(K,N).
\end{align*}
Since $\dim_K\Ext^{-1}_S(K,M)=\dim_K\Ext^{-1}_R(K,N)=0$, we know that $\dim_K\Ext^0_S(K,M)=\dim_K\Ext^0_{R} (K,N)$ as well.  Inductively, $\dim_K\Ext^i_S(K,M)=\dim_K\Ext^i_{R} (K,N)$ for all $i \in \NN$.
\end{proof}

\begin{Cor}\label{WellDefCM}
Let  $(V,pV,K)$ be a complete DVR of unramified mixed characteristic $p>0$, and let $S=V[[x_1\ldots, x_n]]$.
Let $I$ be an ideal of $S$ such that $S/I$ is a Cohen-Macaulay ring of characteristic $p$. 
Then for all $i, j \in \NN$,
$$\dim_K\Ext^{i}_{S/pS} (K,H^{n-j}_{I S/pS} (S/pS))=\dim_K\Ext^{i}_S(K,H^{n+1-j}_{I} (S)).$$
\end{Cor}
\begin{proof}
Let $R=S/pS$. The short exact sequence $0\to S\FDer{\cdot p} S\to R\to 0$ induces the short exact sequence
$$
0\to H^{n-d}_I(R)\to
H^{n-d+1}_I(S)\FDer{\cdot p} 
H^{n-d+1}_I(S)\to 0
$$
since $H^{n-d+1}_I(R) = 0$ by \cite[Proposition 4.1]{P-S}.
Since $H^{n-d}_I(S)=0$   $H^{i}_I(S)\FDer{p} H^{i}_I(S)$ is injective for $i\neq n-d+1$, and $H^{n-d}_I(S)=0$. 
The result then follows from Lemma \ref{LemmaExtSurj}.
\end{proof}

\begin{Prop}\label{AgreeDimTwo}
Let $(V,pV,K)$ be a complete DVR of unramified mixed characteristic $p>0$, and let $S=V[[x_1\ldots, x_n]]$.
Let $I$ be an ideal of $S$ containing $p$, such that $\dim(S/I)\leq 2$.
Then 
$$
\dim_K\Ext^{d}_{S/pS} (K,H^{n-d}_{I S/pS} (S/pS))=\dim_K\Ext^{d}_S(K,H^{n+1-d}_{I} (S)).
$$
\end{Prop}
\begin{proof}
Let $R=S/pS$.  Consider the following cases.

If $\dim(S/I)=0$, $H^{n+1}_I (S)=E_S(K)$ and $H^{n}_{I} (S)=E_R(K)$. Then
$$
\dim_K\Ext^{d}_{S/pS} (K,H^{n-d}_{I S/pS} (S/pS))=\dim_K\Ext^{d}_S(K,H^{n+1-d}_{I} (S))=1.
$$  

If $\dim(S/I)=1$, the short exact sequence $0\to S\FDer{\cdot p} S\to R\to 0$ induces a long exact sequence
$
0\to H^{n-1}_I(R)\to H^{n}_I (S)\FDer{\cdot p} H^{n}_I (S)\to 0 
$
by 
the Hartshorne-Lichtenbaum vanishing theorem. The proposition then follows from Lemma \ref{LemmaExtSurj}.

Suppose that $\dim(S/I)=2.$
First assume that $I $ has pure dimension $2$.
Let $\alpha$ be the number of connected components of  
$\Spec (\widehat{A})\setminus \{m\}$, where $A=\widehat{R/I}^{sh}$ is the strict Henselization of $R/I$. 
In fact, $\alpha=\dim_K\Ext^{2}_{R} (K,H^{n-2}_{I} (S/pS))$ (cf.\ \cite[Proposition $2.2$]{Walther2}).

We prove the statement by induction on $\alpha$.
If $\alpha=1$, the short exact sequence $0\to S\FDer{\cdot p} S\to R\to 0$ induces the short exact 
sequence
$$
0\to H^{n-2}_I (R)\to H^{n-1}_I (S)\FDer{\cdot p} H^{n-1}_I (S) \to 0,
$$
since $H^{n-1}_I(R)=0$ by \cite[Theorem $2.9$]{H-L}.  The proposition then follows from Lemma \ref{LemmaExtSurj}.
If $\alpha>1$, we pick ideals $J_1,\ldots, J_\alpha$ such that $I=J_1\cap\ldots \cap J_\alpha$, and each $J_k$ defines a connected component of
$\Spec (\widehat{A})\setminus \{m\}$. Let $J$ denote $ J_1\cap\ldots \cap J_{\alpha-1}$. Using the Mayer-Vietoris sequence, we obtain
an isomorphism
$$
H^{n-1}_{J}(S)\oplus H^{n-1}_{J_\alpha}(S)\cong  H^{n-1}_{I}(S)
$$
because $\sqrt{J+J_\alpha}=m$. Then 
\begin{align*} \dim_K \Ext^2_S(K,H^{n-1}_{I}(S)) = \dim_K \Ext^2_S(K,H^{n-1}_{J}(S))+ \dim_K\Ext^2_S(K,H^{n-1}_{J_\alpha}(S)) =\alpha.\end{align*}
By Lemma \ref{LCVPureTwo}, \cite[Lemma $1.4$]{LyuDmodules}  and \cite[Proposition $2.2$]{Walther2}, the other numbers are determined by $\alpha$.

For the general case such that $\dim(S/I)=2,$
let $P_1,\ldots, P_r$ be the minimal primes of dimension one of $I$, and let
$Q_1,\ldots, Q_s$ be the minimal primes of dimension two of $I$.
Let $J_1=P_1\cap\ldots \cap P_r$ and $J_2=Q_1\cap\ldots \cap Q_s$.
We claim that $\Ext^j_S(K,H^{n-1}_{I}(S))=\Ext^j_S(K,H^{n-1}_{J_2}(S)).$
Let $f_1,\ldots,f_\ell\in J_2\setminus I$ such that $I +(f_1,\ldots,f_\ell)S=J_2.$
We proceed by induction on $\ell$; first assume that $\ell=1.$
Since $H^{n-1}_{I}(S)=H^{n-1}_{J_2}(S),$ $H^{n-1}_I(S_{f_1})=0.$
The long exact sequence
\[ \xymatrix@R=.65cm{ \ar[r] 
0 & H^{n-1}_{I+f_1S}(S) \ar[r] & H^{n-1}_I(S) \ar[r] & H^{n-1}_I(S_{f_1})  \connarrow & \\ 
& H^{n}_{I+f_1S}(S)  \ar[r] & H^{n}_I(S) \ar[r] & H^{n}_I(S_{f_1})  \ar[r] & 0, \\ 
}\]
then indicates both that $ H^{n-1}_{I+f_1S}(S)\cong H^{n-1}_I(S)$, and that
$$ 0\to H^{n}_{I+f_1S}(S)\to H^{n}_I(S)\to H^{n}_I(S_f)\to 0$$ is exact.
Hence, $\Ext^j_S(K,H^{n-1}_{I}(S))=\Ext^j_S(K,H^{n-1}_{I+f_1S}(S)).$
Moreover, $I+f_1S\subseteq J_2$ is an ideal of dimension $2$, whose minimal primes of dimension $2$ are $P_1,\ldots, P_r.$ 
If we assume that the claim is true for $\ell,$ the proof for $\ell+1$ is analogous to the previous part.
\end{proof}

\begin{Cor}\label{Agree}
Let $(R,m,K)$ be a local ring of characteristic $p>0.$
If $R$ is a Cohen-Macaulay ring or if $\dim R\leq2,$ then for $i,j \in \NN$, $\newLyu_{i,j}(R) = \lambda_{i,j}(R)$.
\end{Cor}
\begin{proof}
Since dimension, Cohen-Macaulayness, and both Lyubeznik numbers are preserved after completion, we can assume that $R$ is complete.
Then the result follows from Corollary \ref{WellDefCM} and Proposition \ref{AgreeDimTwo}.
\end{proof}
\section{An example for which the equal-characteristic and the $\newLyuName$ differ}\label{DifferentNum}

Throughout this section, we will often refer to the following ring and ideal:

\begin{Notation}\label{Notation}
Let $R=\ZZ_Q[x_1,\ldots,x_6]$, where  $p=2$ and $Q=p\ZZ$. 
Moreover, let $I$ denote the ideal of $R$ generated by the 11 elements
$$p, x_1x_2x_3 , x_1x_2x_4 , x_1x_3x_5 ,x_1x_4x_6 , x_1x_5x_6 ,$$
 $$x_2x_3x_6 ,x_2x_4x_5 , x_2x_5x_6 , x_3x_4x_5 , x_3x_4x_6.$$
\end{Notation}

\begin{Rem}\label{Rem1}
It is easily checked that for $I \subseteq R$ as in Notation \ref{Notation}, $\Depth_{I}(R)=4$.  Thus, the short exact sequence $0\to R\to R_p\to R_p/R\to 0$ induces the long exact sequence
\begin{equation}\label{LESL}
0\to H^3_I(R_p/R) \to H^4_I(R) \to H^4_I(R_p)\to H^4_I(R_p/R)\to \cdots.
\end{equation}
Since $H^i_I(R)$ is supported at $p\in I$, for all $i \in \NN$, $H^i_I(R_p)=0$, so $H^i_I(R_p/R) \cong H^{i+1}_I(R)$.
\end{Rem}

\begin{Prop}\label{NotSurj}
With $R,p,$ and $I$ as in Notation \ref{Notation},  the map
$$
H^{3}_I(R_p /R) \FDer{\cdot p} H^{3}_I(R_p /R)
$$ 
is not surjective.
\end{Prop}
\begin{proof}  
Since $\Depth_{I}(R)=4,$ $H^{0}_I(R_p/R)=H^{1}_I(R_p/R)=H^{2}_I(R_p/R)=0$ by the long exact sequence in local cohomology (see Remark \ref{Rem1}).  For every $\ell\in\NN$, the exact sequence
$0\to R/p^\ell R \to R_p/R \FDer{p^\ell } R_p/R\to 0$
induces a long exact sequence
\begin{equation}\label{LESQ}
0\to H^{3}_I(R/p^\ell R)\to H^{3}_I(R_p/R)\FDer{p^\ell} H^{3}_I(R_p/R)\FDer{\partial} 
H^{4}_I(R/p^\ell R)\to  \cdots.
\end{equation}
Note that $H^{3}_I(R/p^\ell R)=\Ann_{H^{3}_I(R_p/ R)}(p^\ell R).$

As the direct limit functor is exact, 
the limit of the direct system of short exact sequences in Figure \ref{Figure1} is the short exact sequence $0\to R/pR\to R_p/R \FDer{\cdot p} R_p/R\to 0.$
 Moreover, $ H^j_I(R_p/R)=\lim \limits_{\longrightarrow} H^j_I(R/p^\ell R)$.

\begin{figure}[h!]
{
\[ \xymatrix@R=.65cm{ \ar[r] 
0 & R/pR \ar[d]_= \ar[r]^{\cdot p} & R/p^2 R \ar[d]_{\cdot p} \ar[r] & R/pR  \ar[d]_{\cdot p} \ar[r] & 0 \\ 
0 \ar[r] & R/pR \ar[d]_= \ar[r]^{\cdot p^2} & R/p^3 R \ar[d]_{\cdot p} \ar[r] & R/p^2 R  \ar[d]_{\cdot p} \ar[r] & 0 \\ 
0 \ar[r] & R/pR \ar[d] \ar[r]^{\cdot p^3} & R/p^4 R \ar[d] \ar[r] & R/p^3 R  \ar[d] \ar[r] & 0 \\
& \vdots & \vdots & \vdots & 
}\]
}
\begin{center}
\caption{}
\label{Figure1}
\end{center}
\end{figure}

%
%
%
%
%
%
%
%
%
%
%

By \cite[Example $5.10$]{Bockstein}, the Bockstein homomorphism 
$ H^{3}_I(R/p^\ell R)\to H^{4}_I(R/p^\ell R)$ is nonzero, so  
$H^{3}_I(R/p^2 R) \FDer{\pi}  H^{3}_I(R/p R)$ is not surjective
by the isomorphism of sequences given in Figure \ref{Figure2}. Therefore,
$\Ann_{H^{3}_I(R_p /R)} p^2R \FDer{\cdot p} \Ann_{H^{3}_I(R_p /R)} pR$ is not surjective, so that
$H^{3}_I(R_p /R) \FDer{\cdot p} H^{3}_I(R_p /R)$ is also not surjective.

\begin{figure}[h!]
{
\[ \xymatrix@R=.65cm{ \ar[r] 
0 & H^{3}_I(R/p R) \ar[d] \ar[r]^{\cdot p} & H^{3}_I(R/p^2 R)  \ar[d] \ar[r]^{\pi} & H^{3}_I(R/p^2 R)   \ar[d] \ar[r] & 0 \\ 
0 \ar[r] & \Ann_{H^{3}_I(R_p / R)} pR \ar[r] &  \Ann_{H^{3}_I(R_p /R)} p^2R \ar[r]^{\cdot p} & \Ann_{H^{3}_I(R_p /R)} pR  \ar[r] & 0 \\ 
}\]
}
\begin{center}
\caption{}
\label{Figure2}
\end{center}
\end{figure}

%
%
%
%

\end{proof}

\begin{Rem}\label{Koszul}
Let $R=\FF_2[y_1,\ldots,y_5]$, and let $J=(y_1 y_2, y_2 y_3,y_3 y_4, y_4 y_5,y_5 y_1).$
Then 
$J=(y_2,y_3,y_5)\cap (y_1,y_3,y_4)\cap (y_1,y_2,y_4)\cap (y_1,y_3,y_5)\cap (y_2,y_4,y_5).$
Now, $R/J$ is a graded Cohen-Macaulay ring of dimension $2$, where the classes of
$y_1+y_2+y_3$ and $y_1+y_4+y_5$ form a homogeneous system of parameters.
Then $H^i_J(R)\neq 0$ if only if $i=3$ \cite[Proposition 4.1]{P-S}.  (See \cite[Proposition $3.1$]{MontanerAdv} for an analog in characteristic zero.)
\end{Rem}
\begin{Lemma}\label{Supp}
Consider $R$, $p$, and $I$ as in Notation \ref{Notation}. Then
$ H^{4}_I\left(R/pR\right)$ is supported only at the maximal ideal
$(p,x_1,\ldots , x_6)$.
\end{Lemma}
\begin{proof}
Let $\overline{R}$ denote $R/pR.$
Every associated prime in $\Ass_S H^{4}_I (\overline{R})$
has the form $(2,x_{i_1},\ldots, x_{i_j})R$ for some $\{i_1,\ldots i_j\}\subseteq \{1,\ldots, 6\}$
by \cite[Proposition 2.5]{YSq} and \cite[Proposition $2.7$]{YStr}, as $I\overline{R}$ is a square-free monomial ideal.
Then it suffices to prove that $H^{4}_I (\overline{R})_{x_i}=0$ for every $1 \leq i\leq 6.$ 

We first check that $H^{4}_I (\overline{R})_{x_6}=0$. Let $A=\FF_2[x_1,\ldots,x_5]$, and
$$J=(x_1x_2 ,x_2x_3, x_3x_4 ,x_4 x_5, x_5 x_1  )\subseteq A.$$ Now, $A$ and $J$ are as in Remark \ref{Koszul}, so that
$ H^{4}_J(A)=0$. Since $A[x_6]_{x_6}=\overline{R}_{x_6}$ is a flat extension,
$H^{4}_{J} (\overline{R})_{x_6} = H^{4}_{J}(A)\otimes_A \overline{R}_{x_6}=0$.
Note that $J \overline{R}_{x_6}=I\overline{R}_{x_6}$, and so $H^{4}_{I} (\overline{R})_{x_6}= H^{4}_{J}(\overline{R})_{x_6}=0$. 
The proof that $H^{4}_I (\overline{R})_{x_i}=0$ for $1 \leq i\leq 5$ is analogous, and again relies on Remark \ref{Koszul}.
\end{proof}

\begin{Cor}\label{H4}
Take $R$, $p$, and $I$ as in Notation \ref{Notation}. Let $S =
\widehat{R}_m$, where $m$ is the maximal ideal $(p,x_1,\ldots,x_6)$ of $R$.
Then
$$\Coker\left( H^3_I(S_p/S)\FDer{\cdot p} H^3_I(S_p/S)\right)\cong H^4_I(S/pS)\cong E_{S/pS}(\FF_2).$$
\end{Cor}
\begin{proof}
Note that $S=\widehat{\ZZ}_{Q}[[x_1,\ldots, x_6]]$, the power series ring with coefficients in the $p$-adic integers.
Then $S/pS \cong \FF_2[[x_1,\ldots, x_6]]$, and  $H^{4}_I (S/pS)$ is a $D(S/pS,\FF_2)$-module supported only at $m$ by Lemma \ref{Supp}.  
Thus, $H^{4}_I (S/pS)=E_{S/pS}(\FF_2)$ by \cite[Example $4.6$]{LyuNumMontaner}.
Since $\Coker\left(H^3_I(S_p/S)\FDer{\cdot p} H^3_I(S/pS)\right)$ is nonzero by Proposition \ref{ContraExSurj}, and injects into 
$H^{4}_I (S/pS)=E_{S/pS}(\FF_2)$ by  the long exact sequence 
$$
0\to H^3_I(S/pS)\to H^3_I(S_p/S)) \FDer{\cdot p} H^3_I(S_p/S)\to H^4_I(S/pS))\to\ldots,
$$
we know that $\Coker\left(H^3_I(S_p/S)\FDer{\cdot p} H^3_I(S/pS)\right)$ is a nonzero $D(S/pS,\FF_2)$-submodule of $E_{S/pS}(\FF_2)$. Hence, 
$$\Coker\left( H^3_I(S_p/S)\FDer{\cdot p} H^3_I(S_p/S)\right)\cong H^4_I(S/pS)\cong E_{S/pS}(\FF_2),$$
since $E_{S/pS}(\FF_2)$ is a simple $D(S/pS,\FF_2)$-module.
\end{proof}

\begin{Cor}\label{ContraExSurj}
There exists a regular local ring $S$ of unramified mixed characteristic $p>0$, with an ideal $I$ of $S$ containing $p$, so that the map
$$
H^{4}_I (S) \FDer{\cdot p} H^{4}_I(S)
$$
is not surjective.
\end{Cor}
\begin{proof}
Again, consider $R$, $p$, and $I$ as in Notation \ref{Notation}.
Let $S=\widehat{R}_m$, where $m = (p,x_1,\ldots,x_6)R$.
Then by Corollary \ref{H4}, $\Coker\left(H^{4}_I (S) \FDer{\cdot p} H^{4}_I(S)\right)\cong E_{S/pS}(\FF_2) \neq 0.$
\end{proof}

\begin{Prop} \label{LyuTableCE}
Take $R$, $p$, and $I$ as in Notation \ref{Notation}. Let $S =
\widehat{R}_m$, where $m = (p,x_1,\ldots,x_6).$
Then \[ \widetilde{\Lambda}(S/IS) = \begin{pmatrix}
 0 & 0 & 0 & 0 \\
 0 & 0 & 0 & 0 \\
 0 & 0 & 0 & 0 \\
 0 & 0 & 0 & 1 \\
\end{pmatrix}.\] 
\end{Prop}
\begin{proof}
We will abuse notation, using ``$I$" to mean $IS$.  
By Corollary \ref{H4} and 
Proposition \ref{NotSurj}, we have the short exact sequence
$$
0\to H^3_I(S/pS)\to H^4_I(S)\FDer{\cdot p} H^4_I(S)\to H^4_I(S/pS)\to 0.
$$ 
We have that $H^4_I(S)\neq 0$ and all other $H^i_I(S)$ must vanish, since the only way that $H^i_I(S)\FDer{\cdot p} H^i_I(S)$ is injective is if $H^i_I(S)=0$.
Since $H^j_i(S)\neq 0$ if and only if $j=4$, the spectral sequence
$$
E^{p,q}_2=H^p_m(H^q_I(S))\Longrightarrow H^{p+q}_{m}(S) = E^{p,q}_{\infty}
$$
converges at the second stage, so that
$H^3_m H^4_I(S) \cong E_S (K)$, and all other $H^j_m H^i_I(S)$ vanish. 
Since all $H^j_m  H^i_I (S)$  are injective $S$-modules, 
$H^j_m H^i_I (S) \cong E_S(K)^{\bigoplus \lambda_{i,j}(S)}$ as in the proof of \cite[Lemma 1.4]{LyuDmodules}, and the result follows.

\end{proof}

\begin{Rem} \label{RmkDiff}
Together, Proposition \ref{LyuTableCE} and \cite[Example 4.6]{LyuNumMontaner} show that the $\newLyuName$ are not always the same as the equal characteristic ones. 
Take $I\subseteq R$ as defined in Notation \ref{Notation}. Let $S_1=\FF_2[[x_1,\ldots,x_6]]$, and $S_2=\widehat{\ZZ}_Q[[x_1,\ldots,x_6]],$ where $Q=2\ZZ.$  By \cite[Example 4.6]{LyuNumMontaner}, the Lyubeznik numbers in equal characteristic $2$ are given by
\[ \Lambda(S_1/IS_1) =\begin{pmatrix}
 0 &  0 & 1 & 0 \\
 0 &  0 & 0 & 0 \\
 0 &  0 & 0 & 1 \\
 0 &  0 & 0 & 1 \\
\end{pmatrix}.\]
On the other hand, by Proposition \ref{LyuTableCE}, the $\newLyuName$ in
in mixed characteristic $2$, 
\[ \widetilde{\Lambda}(S_2/IS_2) = \begin{pmatrix}
 0 & 0 & 0 & 0 \\
 0 & 0 & 0 & 0 \\
 0 & 0 & 0 & 0 \\
 0 & 0 & 0 & 1 \\
\end{pmatrix}.\] 
In particular, this is an example of a negative answer for Question \ref{QuestionCharp}.
\end{Rem}

The computation in Remark \ref{RmkDiff} is related to work in \cite{MontanerAdv}.

\begin{Teo}\label{TeoCE}
There exists a regular local ring $S$ of unramified mixed characteristic $p>0$, and an ideal $I$ of $S$, such that for some $i,j \in \NN$, $$\dim_K\Ext^{j}_S(K,H^{i}_{I} (S))\neq \dim_K\Ext^{j}_{S/pS} (K,H^{i-1}_{IS/pS} (S/pS)).$$
\end{Teo}
\begin{proof}
Take $R$, $p$, and $I$ as in Notation \ref{Notation}. Let $S$ denote
$\widehat{R}_m$, where $m= (p,x_1,\ldots,x_6).$
Then
 $\dim_K\Ext^{0}_R(K,H^{5}_{I} (S))=0\neq 1=\dim_K\Ext^{0}_{S/pS} (K,H^{4}_{IS/pS} (S/pS)).$
\end{proof}

\section{Further Questions}

\begin{Question}
If $R$ is a $d$-dimensional Cohen-Macaulay local ring containing a field, Kawasaki showed that $\lambda_{d,d}(R)=1$ \cite{Kawasaki}.    If $R$ is a $d$-dimensional Cohen-Macaulay local ring of mixed characteristic, is $\newLyu_{d,d}(R) = 1$?
\end{Question}

\begin{Question}
Does there exist a local ring $R$ of equal characteristic $p>0$ such that for some $i,j\in\NN$, $\lambda_{i,j}(R) \neq \newLyu_{i,j}(R)$, and both are nonzero (cf.\ Remark \ref{RmkDiff}, Theorem \ref{TeoCE})?
\end{Question}

\begin{Question}
Consider a local ring $R$ of characteristic $p>0$ such that $\lambda_{i,j}(R) \neq \newLyu_{i,j}(R)$ for some $i,j\in \NN$.
The vanishing of each invariant in Remark \ref{RmkDiff} suggests that the equal-characteristic invariants might capture some finer information about $R$ than do the $\newLyuName$.  How can we characterize this data?
On the other hand, do the $\newLyuName$ capture properties of $R$ that the equal-characteristic Lyubeznik numbers miss?
\end{Question}

\begin{Question}
Do the $\newLyuName$ have geometric interpretations?
\end{Question}

\section*{Acknowledgments}
We are grateful to Josep \`Alvarez Montaner, 
Daniel Hern\'andez, Mel Hochster, Gennady Lyubeznik, and Felipe P\'erez-Vallejo for their valuable comments and helpful discussions.  
We also thank the $2010$ AMS Mathematical Research Communities program for supporting meetings between the authors.
The first author also thanks the National Council of Science and Technology of Mexico for support through Grant $210916.$

\bibliographystyle{alpha}
\bibliography{References}

\vspace{.3cm}
\small{
{\sc Department of Mathematics, University of Michigan, Ann Arbor, MI $48109$-$1043,$ USA.}

{\it Email address:}  \texttt{luisnub@umich.edu}

\vspace{.3cm}

{\sc Department of Mathematics, University of Minnesota,  Minneapolis, MN $55455,$ USA.}

{\it Email address:}  \texttt{ewitt@umn.edu}
}

\end{document}